\documentclass[12pt]{amsart}
\usepackage{amssymb}
\usepackage[all]{xy}

\usepackage{verbatim}
\usepackage{eucal}
\usepackage{latexsym}
\usepackage{amsfonts,amssymb,amsxtra}
\usepackage{relsize}
\usepackage[dvips]{color}
\usepackage{latexsym}
\usepackage{mathptmx}
\usepackage{amssymb}
\usepackage{graphicx}
\usepackage{amsthm}
\usepackage{amsfonts}
\usepackage{amscd}
\usepackage{bm}
\usepackage[utf8]{inputenc}

\usepackage{verbatim} 
\usepackage{eucal} 
\usepackage{color}

\usepackage[draft]{todonotes}
\usepackage{pgf}
\usepackage{tikz}
\usetikzlibrary{arrows,automata}

\usepackage{pgf, tikz}
\definecolor {processblue}{cmyk}{0.96,0,0,0}
\input xy
\xyoption{all}
\usepackage{mathtools}
\usepackage{amsfonts}
\usepackage{amsthm}
\usepackage{mathrsfs}

\usepackage{amscd}
  \newtheorem{The}{Theorem}[section]
  
  \newtheorem{Lem}[The]{Lemma}
  \newtheorem{Cor}[The]{Corollary}

  \newtheorem{Examp}[The]{Example}

\newcommand{\bsm}{\begin{smallmatrix}}
\newcommand{\esm}{\end{smallmatrix}}
\newcommand{\bbm}{\begin{matrix}}
\newcommand{\ebm}{\end{matrix}}

\newcommand{\Hom}{\rm{Hom}}

\theoremstyle{definition}

\theoremstyle{plain}

\theoremstyle{definition}

\numberwithin{equation}{section}
\newtheorem*{theorem a*}{Theorem A}
\newtheorem*{theorem b*}{Theorem B}
\newtheorem*{theorem c*}{Theorem C}
\newtheorem*{corollary a*}{Corollary A}
\newtheorem*{corollary b*}{Corollary B}
\newtheorem*{corollary c*}{Corollary C}
\begin{document}

\title{Finiteness and purity of subcategories of the module categories}

\author{Ziba Fazelpour}
\address{School of Mathematics, Institute for Research in Fundamental Sciences (IPM), P.O. Box: 19395-5746, Tehran, Iran}
\email{z.fazelpour@ipm.ir}
\author{Alireza Nasr-Isfahani}
\address{Department of Pure Mathematics\\
Faculty of Mathematics and Statistics\\
University of Isfahan\\
P.O. Box: 81746-73441, Isfahan, Iran\\ and School of Mathematics, Institute for Research in Fundamental Sciences (IPM), P.O. Box: 19395-5746, Tehran, Iran}
\email{nasr$_{-}$a@sci.ui.ac.ir / nasr@ipm.ir}



\begin{abstract}
In this paper, by using functor rings and functor categories, we study finiteness and purity of subcategories of the module categories. We give a characterisation of contravariantly finite resolving subcategories of the module category of finite representation type in terms of their functor rings. We also characterize contravariantly finite resolving subcategories of the module category $\Lambda${\rm -mod} that contain the Jacobson radical of $\Lambda$ of finite type, by their functor categories. We study the pure semisimplicity conjecture for a locally finitely presented category ${\underrightarrow{\lim}}\mathscr{X}$ when $\mathscr{X}$ is a covariantly finite subcategory of $\Lambda$-mod and every simple object in Mod$(\mathscr{X}^{\rm op})$ is finitely presented and give a characterization of covariantly finite subcategories of finite representation type in terms of decomposition properties of their closure  under filtered colimits. As a consequence we study finiteness and purity of $n$-cluster tilting subcategories and the subcategory of the Gorenstein projective modules of the module categories. These results extend and unify some known results of \cite{aus, ausla2, ab, che, EN, fu, fre, ht, wis}.
\end{abstract}

\maketitle

\section{Introduction}
A left artinian ring $\Lambda$ is called {\it of finite representation type} if it has, up to isomorphism, only finitely many finitely generated indecomposable left $\Lambda$-modules. A ring $\Lambda$ is called {\it left pure semisimple} if every left $\Lambda$-module is a direct sum of finitely generated left $\Lambda$-modules \cite{aus}. It is known that a ring $\Lambda$ is of finite representation type if and only if $\Lambda$ is right and left pure semisimple \cite{aus}. The problem of whether left pure semisimple rings are of finite representation type, known as the pure semisimplicity conjecture, still remains open (see \cite{aus}, \cite{Sim6} and \cite{Sim7}). In representation theory of artin algebras, the classification of algebras of finite representation type are of particular importance since for this class of algebras one has a complete combinatorial description of the module category in terms of the Auslander-Reiten quiver.
Certain classes of subcategories of the module category plays an important role in the representation theory of artin algebras. Functorially finite subcategories (in the
sense of Auslander-Smalo \cite{au}), contravariantly finite resolving subcategories, the subcategory of Gorenstein projective modules, $n$-cluster tilting subcategories, $\cdots$ are typical examples of such subcategories. An artin algebra $\Lambda$ is called {\it of finite CM-type} if there are only finitely many, up to isomorphism, indecomposable finitely generated Gorenstein projective $\Lambda$-modules. Chen in \cite[Main theorem]{che} proved that a Gorenstein artin algebra $\Lambda$ is of finite CM-type if and only if every Gorenstein projective $\Lambda$-module is a direct sum of finitely generated Gorenstein projective modules. Beligiannis in \cite{ab} generalized and extended Chen's result. A Krull-Schmidt subcategory $\mathscr{C}$ of $\Lambda$-mod is called {\it of finite representation type} or {\it of finite type} if it has only finitely many non-isomorphic indecomposable modules. Beligiannis showed that $\Lambda$ is a virtually Gorenstein artin algebra of finite CM-type if and only if $\Lambda$ is a virtually Gorenstein artin algebra and every Gorenstein projective left $\Lambda$-module is a direct sum of indecomposable left $\Lambda$-modules if and only if every Gorenstein projective left $\Lambda$-module is a direct sum of finitely generated left $\Lambda$-modules (see \cite[Theorem 4.10]{ab})). In this paper we study finiteness and purity properties of covariantly finite (and hence functorially finite) subcategories and contravariantly finite resolving subcategories of the module category over a left artinian ring.

First, we give the following characterisation of contravariantly finite resolving subcategories of the module category of finite representation type. We refer the reader to Section 2 for the definitions.\\

\textbf{Theorem A.} (Theorem \ref{2.1})
Assume that $\Lambda$ is a left artinian ring and $\mathscr{X}$ is a contravariantly finite resolving subcategory of $\Lambda${\rm -mod}. Let $\lbrace V_{\alpha}~|~\alpha \in J\rbrace$ be a complete set of representative of the isomorphic classes of indecomposable modules in $\mathscr{X}$ and $T=\widehat{{\rm End}}_{\Lambda}(V)$, where $V=\bigoplus_{\alpha \in J}V_{\alpha}$.
\begin{itemize}
\item[$(1)$] The following statements are equivalent:
\begin{itemize}
\item[$({\rm i})$] $\mathscr{X}$ is of finite representation type and ${\rm res.dim}_{\mathscr{X}}\Lambda${\rm -mod} $\leq 2$.
\item[$({\rm ii})$] $\mathscr{X}$ is of bounded representation type and ${\rm res.dim}_{\mathscr{X}}\Lambda${\rm -mod} $\leq 2$.
\item[$({\rm iii})$] $T$ is a right perfect ring and ${\rm l.gl.dim}T \leq 2$.
\item[$({\rm iv})$] $T$ is a left locally finite ring and ${\rm l.gl.dim}T \leq 2$.
\end{itemize}
\item[$(2)$] The following statements are equivalent:
\begin{itemize}
\item[$({\rm i})$]  $T$ is a right locally noetherian ring and every module in $\mathscr{A}={\underrightarrow{\lim}}\mathscr{X}$ is a direct sum of finitely generated indecomposable modules.
\item[$({\rm ii})$] $T$ is a right locally finite ring.
\item[$({\rm iii})$] $V_T$ has finite length.
\end{itemize}
\item[$(3)$] If $T$ is a right locally finite ring and ${\rm l.gl.dim}T \leq 2$, then $\mathscr{X}$ is of finite representation type and ${\rm res.dim}_{\mathscr{X}}\Lambda${\rm -mod} $\leq 2$. Also the converse holds when ${\rm res.dim}_{\mathscr{X}}\Lambda${\rm -mod} $= 0$.
\end{itemize}

When a contravariantly finite resolving subcategory $\mathscr{X}$ of the module category $\Lambda${\rm -mod} contains the Jacobson radical of $\Lambda$ we prove the following theorem.\\

\textbf{Theorem B.} (Theorem \ref{2.2})
Let $\Lambda$ be a left artinian ring and $\mathscr{X}$ be a contravariantly finite resolving subcategory of $\Lambda$-{\rm mod} which contains $J(\Lambda)$. Then the following statements are equivalent.
\begin{itemize}
\item[$({\rm i})$] $\mathscr{X}$ is of finite representation type.
\item[$({\rm ii})$] Any family of homomorphisms between indecomposable modules in $\mathscr{X}$ is both noetherian and artinian.
\item[$({\rm iii})$] ${\rm Mod}(\mathscr{X})$ is locally finite.
\item[$({\rm iv})$] $\mathscr{X}(-,X)$ is finite for each $X \in \mathscr{X}$.
\item[$({\rm v})$] $\mathscr{X}(-,S)$ is finite for each simple left $\Lambda$-module $S$.
\end{itemize}

We also characterize covariantly finite subcategories of the module category of finite representation type. More precisely we prove the following result.\\

\textbf{Theorem C.} (Theorem \ref{3.1})
Let $\Lambda$ be a left artinian ring and $\mathscr{X}$ be a covariantly finite subcategory of $\Lambda$-mod. Assume that every simple $\mathscr{X}^{\rm op}$-module is finitely presented and $\mathscr{A}={\underrightarrow{\lim}}\mathscr{X}$. Then the following statements are equivalent.
\begin{itemize}
\item[$({\rm i})$] $\mathscr{X}$ is of finite representation type.
\item[$({\rm ii})$] Every module in $\mathscr{A}$ is a direct sum of finitely generated modules.
\item[$({\rm iii})$] $\mathscr{A}$ is pure semisimple.
\item[$({\rm iv})$] Any family of homomorphisms between indecomposable modules in $\mathscr{X}$ is noetherian.
\item[$({\rm v})$] ${\rm Mod}(\mathscr{X}^{\rm op})$ is locally finite.
\item[$({\rm vi})$] $\mathscr{X}(X,-)$ is finite for each $X \in \mathscr{X}$.
\item[$({\rm vii})$] $\mathscr{X}(S,-)$ is finite for each simple left $\Lambda$-module $S$.
\item[$({\rm viii})$] There is an additive equivalence between $\mathscr{A}$ and the category of projective modules over a semiperfect ring.
\item[$({\rm ix})$] There is an additive equivalence between $\mathscr{X}$ and the category of finitely generated projective modules over a semiperfect ring.
\end{itemize}
If $\Lambda \in \mathscr{X}$, then $({\rm i})$-$({\rm ix})$ are equivalent to
\begin{itemize}
\item[$({\rm x})$] Every module in $\mathscr{A}$ is a direct sum of indecomposable modules.
\end{itemize}

As a consequence, we give a characterization for functorially finite subcategories of the module category of finite type (see Corollary \ref{3.2}). Also we characterize virtually Gorenstein artin algebras of finite CM-type (see Corollary \ref{3.3}).

Higher dimensional Auslander-Reiten theory was introduced by Iyama in \cite{I2, I1}. It deals with
$n$-cluster tilting subcategories of module categories. The question of finiteness of n-cluster tilting subcategories for $n\geq 2$,
which is among the first that have been asked by Iyama [I3], is still open. Up to now, all known n-cluster tilting subcategories with $n\geq 2$
are of finite type. Several equivalence conditions for purity and finiteness of n-cluster tilting subcategories are given in \cite{DN, EN}. As a consequence of Theorem C we characterize $n$-cluster tilting subcategories of the module categories of finite type (see Corollary \ref{4.1}).

The paper is organized as follows. In Section 2, we recall some definitions and collect several preliminary notions and some known results that will be needed throughout the paper. In Section 3, we study contravariantly finite resolving subcategories of the module category and prove Theorem A. In case a contravariantly finite resolving subcategory $\mathscr{X}$ of $\Lambda$-mod contains the Jacobson radical $J(\Lambda)$ of $\Lambda$ we study finiteness of this subcategory in terms of Mod$(\mathscr{X})$ and prove Theorem B. In Section 4, we study the pure semisimplicity conjecture for locally finitely presented categories ${\underrightarrow{\lim}}\mathscr{X}$ when $\mathscr{X}$ is a covariantly finite subcategory of $\Lambda$-mod and every simple object in Mod$(\mathscr{X}^{\rm op})$ is finitely presented and prove Theorem C.

\subsection{Notation }
{\rm  Throughout this paper all rings are associative with unit unless otherwise stated. Let $R$ be a ring (not necessary with unit). We denote by $R$-Mod (resp., Mod-$R$) the category of all left (resp., right) $R$-modules and by $J(R)$ the Jacobson radical of $R$. We denote by $R$-mod the category of all finitely generated left $R$-modules. A left (resp., right) $R$-module $M$ is called {\it unitary} if $RM=M$ (resp., $MR=M$). We denote by $R$Mod the category of all unitary left $R$-modules. We denote by ${\rm Proj}(R)$ (resp., ${\rm proj}(R)$) the full subcategory of $R$Mod  whose objects are projective (resp., finitely generated projective) unitary left  $R$-modules. Also we denote by ${\rm Inj}(R)$ the full subcategory of $R$Mod whose objects are injective unitary left $R$-modules. Let $\mathscr{C}$ and $\mathscr{C'}$ be two additive categories and $\mathscr{D}$ be a full subcategory of $\mathscr{C}$. If $F: \mathscr{C} \rightarrow \mathscr{C'}$ is a functor, then we denote by $F|_{\mathscr{D}}$ the restriction of $F$ to $\mathscr{D}$. We denote by Mod$(\mathscr{C})$ (resp., Mod$(\mathscr{C}^{\rm op})$) the category of all contravariant (resp., covariant) additive functors from  $\mathscr{C}$ to the category $\mathfrak{Ab}$ of all abelian groups which is called the category of {\it $\mathscr{C}$-modules} (resp.,{\it $\mathscr{C}^{\rm op}$-modules}). We denote by ${\rm Flat}(\mathscr{C})$ the full subcategory of Mod$(\mathscr{C})$ whose objects are flat $\mathscr{C}$-modules; recall that a $\mathscr{C}$-modules $M$ is called flat if $M$ is a direct limit of representable functors. Let $\mathscr{S}$ be a class of objects of $\mathscr{C}$.  We denote by ${\rm Add}(\mathscr{S})$ the full subcategory of $\mathscr{C}$ consisting the direct summands of coproducts of objects of $\mathscr{S}$. When $\mathscr{S}$ consists of only one object $V$, we write ${\rm Add}(V)$ for ${\rm Add}(\mathscr{S})$. Throughout this paper, we assume that $\Lambda$ is a ring and $\mathscr{X}$ is an additive full subcategory of $\Lambda$-mod which is closed under direct summands,  isomorphisms and contains zero modules.}

\section{Preliminaries}

In this section we recall some basic facts and definitions and prove some preliminary results which will be used in the paper.

\subsection{Contravariantly finite resolving subcategories}

Let $\Lambda$ be a left noetherian ring and $M$ be a finitely generated left $\Lambda$-module. A morphism $\theta: X \rightarrow M$ is called a {\it right $\mathscr{X}$-approximation of $M$} if $X \in \mathscr{X}$ and any morphism from an object in $\mathscr{X}$ to $M$ factors through $\theta$. The subcategory $\mathscr{X}$ of $\Lambda$-mod is called {\it contravariantly finite} if every finitely generated left $\Lambda$-module has a right $\mathscr{X}$-approximation. Left $\mathscr{X}$-approximations and covariantly finite subcategories are defined dually. Also $\mathscr{X}$ is called {\it functorially finite} if it is both contravariantly and covariantly finite (see \cite[Page 81]{ausla3}). We recall that the subcategory $\mathscr{X}$ of $\Lambda$-mod is called {\it resolving} provided that $\Lambda \in \mathscr{X}$ and it is closed under extensions and kernels of epimorphisms (see \cite{abr}).\\

Let $\Lambda$ be a left noetherian ring and assume that $\mathscr{X}$ is a contravariantly finite subcategory of $\Lambda$-mod. An {\it $\mathscr{X}$-resolution} of a finitely generated left $\Lambda$-module $M$ is a complex  $$\cdots \rightarrow X_2 \rightarrow X_1 \rightarrow X_0 \rightarrow M \rightarrow 0$$ with each $X_i \in \mathscr{X}$ such that it is exact by applying the functor ${\rm Hom}_{\Lambda}(X,-)$ for each $X \in \mathscr{X}$ (see \cite{ej}). Note that if $\Lambda \in \mathscr{X}$, then each $\mathscr{X}$-resolution of a module $M$ in $\Lambda$-mod is an exact sequence. We recall from \cite[Section 8.4]{ej} that an {\it $\mathscr{X}$-resolution dimension} ${\rm res.dim}_{\mathscr{X}}M$ of a module $M$ in $\Lambda$-mod is defined to be the minimal integer $n \geq 0$ such that there is an $\mathscr{X}$-resolution $0 \rightarrow X_n \rightarrow \cdots \rightarrow X_0 \rightarrow M \rightarrow 0$. If there is no such an integer, we put ${\rm res.dim}_{\mathscr{X}}M=\infty$. The {\it global $\mathscr{X}$-resolution dimension} ${\rm res.dim}_{\mathscr{X}}\Lambda$-mod is defined to be the supremum of the $\mathscr{X}$-resolution dimensions of all finitely generated left $\Lambda$-modules.

\begin{Lem}\label{1.7}
Let $\Lambda$ be a left noetherian ring and $\mathscr{X}$ be a contravariantly finite resolving subcategory of $\Lambda$-mod. If ${\rm res.dim}_{\mathscr{X}}\Lambda$-mod $\leq 2$, then $\mathscr{X}$ is closed under kernel of morphisms.
\end{Lem}
\begin{proof}
Consider the exact sequence $0 \rightarrow {\rm Ker}h \rightarrow X \overset{h}{\rightarrow} X' \rightarrow {\rm Coker}h \rightarrow 0$  with $X, X' \in \mathscr{X}$. Since ${\rm res.dim}_{\mathscr{X}}\Lambda${\rm -mod} $\leq 2$, ${\rm res.dim}_{\mathscr{X}}{\rm Coker}h \leq 2$. It is easy to see that ${\rm Ker}h \in \mathscr{X}$ when ${\rm res.dim}_{\mathscr{X}}{\rm Coker}h = 0$. Assume that ${\rm res.dim}_{\mathscr{X}}{\rm Coker}h = 1$. Then there is an exact sequence $0 \rightarrow Y_1 \rightarrow Y_0 \rightarrow {\rm Coker}h \rightarrow 0$ with each $Y_i \in \mathscr{X}$. Since the sequence $0 \rightarrow {\rm Im}h \rightarrow X' \rightarrow {\rm Coker}h \rightarrow 0$ is exact, ${\rm Im}h \in \Lambda$-mod and $\mathscr{X}$ is resolving, by \cite[Lemma 2.1]{zx}, ${\rm Im}h \in \mathscr{X}$ and hence ${\rm Ker}h \in \mathscr{X}$. By the similar argument, we can see that ${\rm Ker}h \in \mathscr{X}$ when ${\rm res.dim}_{\mathscr{X}}{\rm Coker}h =2$. Therefore the subcategory $\mathscr{X}$ of $\Lambda$-mod is closed under kernel morphisms.
\end{proof}

\subsection{Functor rings}

Let $\lbrace U_{\alpha}~|~ \alpha \in J \rbrace$ be a family of finitely generated left $\Lambda$-modules. Set $U=\bigoplus_{\alpha \in J} U_{\alpha}$ and for each $\alpha \in J$, letting $e_{\alpha}=\varepsilon_{\alpha}\circ \pi_{\alpha}$, where $\pi_{\alpha}:U \rightarrow U_{\alpha}$ is the canonical projection and $\varepsilon_{\alpha}:U_{\alpha} \rightarrow U$ is the canonical injection. For each left $\Lambda$-module $X$, we define as in \cite[Page 40]{fu}, $\widehat{{\Hom}}_{\Lambda}(U,X)=\lbrace f \in {\rm Hom}_{\Lambda}(U,X)~|~f \circ e_{\alpha} =0~ {\rm for ~almost~ all}~ \alpha \in J \rbrace$.  For $X=U$, we write $\widehat{{\Hom}}_{\Lambda}(U,U)=\widehat{{\rm End}}_{\Lambda}(U)$. Let $T$ be a ring (not necessary with unit). $T$ is called a {\it ring with enough idempotents} if there exists a family $\lbrace q_{\alpha}~|~\alpha \in I \rbrace$ of pairwise orthogonal idempotents of $T$ such that $T=\bigoplus_{\alpha \in I}Tq_{\alpha}=\bigoplus_{\alpha \in I}q_{\alpha}T$ (see \cite[Page 39]{fu}). $R=\widehat{{\rm End}}_{\Lambda}(U)$ is a ring with enough idempotents because of $R=\bigoplus_{\alpha \in J}Re_{\alpha}=\bigoplus_{\alpha \in J}e_{\alpha}R$. Fuller in \cite[Page 40]{fu} defined a covariant functor $\widehat{{\Hom}}_{\Lambda}(U,-):\Lambda$-Mod$\rightarrow R$Mod as follows. For any morphism $f: X \rightarrow Y$ in $\Lambda$-Mod, he defined $\widehat{{\rm Hom}}_{\Lambda}(U,f): \widehat{{\rm Hom}}_{\Lambda}(U,X) \rightarrow \widehat{{\rm Hom}}_{\Lambda}(U,Y)$ via $g \mapsto f \circ g$. From \cite[Pages 40-41]{fu} we observe that the covariant functor $\widehat{{\Hom}}_{\Lambda}(U,-)$ is a left exact functor and preserves direct sums. Moreover $\widehat{{\Hom}}_{\Lambda}(U,-)$ induces an additive equivalence between the full subcategory ${\rm Add}(U)$ of $\Lambda$-Mod and the full subcategory Proj$(R)$ of $R$Mod with the inverse equivalence $U \otimes_R -$. Note that Harada \cite{ha1, ha2} has pointed out that over a ring $T$ with enough idempotents ordinary  direct sums and tensor products behave as they do over a ring with unit. \\

Two rings  with enough idempotents $T$ and $S$ are said to be {\it Morita equivalent} in case there exists an additive equivalence between $R$Mod and $S$Mod (see \cite[Sect. 3]{a}). As in \cite[Page 40]{fu}, we recall that a ring with enough idempotents $R$ is called {\it the functor ring of $\mathscr{X}$} if $R\cong \widehat{{\rm End}}_{\Lambda}(\bigoplus _{\alpha \in J} U_{\alpha})$, where  $\lbrace U_{\alpha}~|~ \alpha \in J \rbrace$ is a complete set of representative of the isomorphic classes of  modules in $\mathscr{X}$.

\begin{Lem}\label{1.5}
Let $\Lambda$ be a left artinian ring and $\lbrace V_{\alpha}~|~ \alpha \in I \rbrace$ be a complete set of representative of the isomorphic classes of indecomposable modules in $\mathscr{X}$. Then $T=\widehat{{\rm End}}_{\Lambda}(V)$ is Morita equivalent to the functor ring $R$ of $\mathscr{X}$, where $V=\bigoplus_{\alpha \in I} V_{\alpha}$.
\end{Lem}
\begin{proof}
Since every finitely generated left $\Lambda$-module is a finite direct sum of indecomposable modules and $\mathscr{X}$ is closed under direct summands, there is an additive equivalence between the full subcategory ${\rm proj}(T)$ of $T$Mod and the full subcategory ${\rm proj}(R)$ of $R$Mod. Hence by \cite[Theorem 3.10]{zn}, $T$ is Morita equivalent to $R$.
\end{proof}

Let $R$ be a ring with enough idempotents. A unitary left $R$-module $U$ is called {\it generator} in $R$Mod if for every pair of distinct morphisms $f,g : M \rightarrow B$ in $R$Mod there exists a morphism $h : U \rightarrow M$ such that $f\circ h \neq g \circ h$ (see \cite[Ch. V, Sect. 7]{s}). $R$ is called {\it left {\rm (resp.,} right{\rm )} locally noetherian} if every finitely generated unitary left (resp., right) $R$-module is noetherian (see \cite[Page 141]{wis}).  It is easy to checked that submodule lattices are ``preserved" by Morita equivalence. This implies that locally noetherian is a Morita invariant property. The ring $R$ has {\it left global dimension less than or equal to 2} which denoted by ${\rm l.gl.dim}R \leq 2$, if the kernel of any homomorphism between projective unitary left $R$-modules is a projective left $R$-module (see \cite[Ch. 7, Sect. 1]{mc}). Note that global dimension property is a Morita invariant property.

\begin{Lem}\label{1.6}
Let $\Lambda$ be a left artinian ring and $\lbrace V_{\alpha}~|~ \alpha \in J \rbrace$ be a family of finitely generated left $\Lambda$-modules. Set $V=\bigoplus_{\alpha \in J} V_{\alpha}$ and $T=\widehat{{\rm End}}_{\Lambda}(V)$. If $V$ is a generator in $\Lambda${\rm -Mod} and ${\rm l.gl.dim}T \leq 2$, then $T$ is a left locally noetherian ring.
\end{Lem}
\begin{proof}
It is enough to show that every $T$-module $\widehat{{\rm Hom}}_{\Lambda}(V,V_{\alpha})$ is noetherian. Let $X$ be a $T$-submodule of $\widehat{{\rm Hom}}_{\Lambda}(V,V_{\alpha})$. Since $V$ is a generator in $\Lambda$-Mod,  $V$ is a finitely generated projective unitary right $T$-module. Since ${\rm l.gl.dim}T \leq 2$ and projective modules in $T$Mod are of the form  $\widehat{{\rm Hom}}_{\Lambda}(V,Q)$ for some $Q \in {\rm Add}(V)$, we have exact sequences $0 \rightarrow Q \overset{g}{\rightarrow} P \rightarrow V_{\alpha}$ and $$0 \rightarrow \widehat{{\Hom}}_{\Lambda}(V,Q) \overset{\widehat{{\Hom}}_{\Lambda}(V,g)}{\rightarrow} \widehat{{\Hom}}_{\Lambda}(V,P) \rightarrow X \rightarrow 0,$$ where $Q, P \in {\rm Add}(V)$. Hence $P/{\rm Im}g$ is a finitely generated left $\Lambda$-module. Since $Q \in {\rm Add}(V)$,  $Q \oplus Q' \cong \bigoplus_{\beta \in B}V_{\beta}$ for some $\Lambda$-module $Q'$. Consider the exact sequence $$0 \rightarrow Q \oplus Q' \overset{(g,id_{Q'})}{\rightarrow} P \oplus Q' \rightarrow P \oplus Q'/{\rm Im}(g,id_{Q'}) \rightarrow 0.$$ It is easy to check that ${\rm Im}g \oplus Q'= \bigoplus_{\beta \in B}W_{\beta}$ with $W_{\beta} \cong V_{\beta}$. By \cite[Proposition 53.2]{wi}, there exists a finitely generated submodule $K$ of $P \oplus Q'$ and a finite subset $B'$ of $B$ such that $P \oplus Q'=(\bigoplus_{\beta \in B\setminus {B'}} W_{\beta}) \oplus L$, where $L=K + \sum_{\beta \in B'}W_{\beta}$. Therefore we have a commutative diagram
\begin{displaymath}
\xymatrix{
0 \ar[r] & \widehat{{\Hom}}_{\Lambda}(V,\bigoplus_{\beta \in B\setminus {B'}} W_{\beta}) \ar[r] \ar[d]_{\widehat{{\Hom}}_{\Lambda}(V, \pi_Q \circ (g^{-1} \oplus id_{Q'}) \circ \ell^{'})} &
\widehat{{\Hom}}_{\Lambda}(V,P \oplus Q') \ar[r]  \ar[d]_{{\widehat{{\Hom}}_{\Lambda}(V,\pi_P)}} & \widehat{{\Hom}}_{\Lambda}(V,L)  \ar [r] & 0\\
0 \ar [r] & \widehat{{\Hom}}_{\Lambda}(V,Q) \ar[r]^{\widehat{{\Hom}}_{\Lambda}(V,g)} & \widehat{{\Hom}}_{\Lambda}(V,P) \ar[r] & X \ar[r]& 0}
\end{displaymath}
where $\ell^{'}: \bigoplus_{\beta \in B\setminus {B'}} W_{\beta} \hookrightarrow {\rm Im}g \oplus Q'$ is the canonical inclusion, $\pi_{Q}: Q \oplus Q' \rightarrow Q$ and $\pi_{P}: P \oplus Q' \rightarrow P$ are the canonical projections. So there is a $T$-module epimorphism $\gamma: \widehat{{\Hom}}_{\Lambda}(V,L) \rightarrow X$. Since $\widehat{{\Hom}}_{\Lambda}(V,L)$ is finitely generated, $X$ is finitely generated and hence $T$ is a left locally noetherian ring.
\end{proof}

 \section{Functor rings of contravariantly finite resolving subcategories}

In this section, we give a characterization of contravariantly finite resolving subcategories $\mathscr{X}$ of $\Lambda$-mod of
finite representation type with ${\rm res.dim}_{\mathscr{X}}\Lambda${\rm -mod} $\leq 2$ in terms of their functor rings. Also we give a characterization of contravariantly finite resolving subcategories $\mathscr{X}$ of $\Lambda$-mod of finite representation type which contain $J(\Lambda)$ in terms of Mod$(\mathscr{X})$. We extend and unify the classical results of Auslander \cite[Theorem 3.1]{ausla2} and Wisbauer \cite[Theorem 3.1]{wis}.\\

Let $\mathscr{A}$ be an additive category with direct limits. An object $A$ in $\mathscr{A}$ is called {\it finitely presented} if the representable functor $\mathscr{A}(A,-)={\rm Hom}_{\mathscr{A}}(A,-): \mathscr{A} \rightarrow \mathfrak{Ab}$ preserves direct limits (see \cite[Page 1642]{wc}). We denote by ${\rm fp}(\mathscr{A})$ the full subcategory of finitely presented objects in $\mathscr{A}$. We recall from \cite{wc} that $\mathscr{A}$ is called {\it locally finitely presented} if ${\rm fp}(\mathscr{A})$ is skeletally small and $\mathscr{A}={\underrightarrow{\lim}}~{\rm fp}(\mathscr{A})$ (i.e., every object in $\mathscr{A}$ is a direct limit of finitely presented objects in $\mathscr{A}$).\\

Let $R$ be a ring with enough idempotents. We recall that $R$ is called {\it left {\rm (resp.,} right{\rm )} locally artinian} if every finitely generated unitary left (resp., right) $R$-module is artinian (see \cite[Page 141]{wis}). $R$ is called  {\it left {\rm (resp.,} right{\rm )} locally finite} if it is both left {\rm (resp.,} right{\rm )} locally noetherian and artinian ring (see \cite[Page 141]{wis}). Moreover $R$ is called {\rm semiperfect} if every finitely generated unitary left $R$-module has a projective cover (see \cite{f}). Also $R$ is called {\it left {\rm (resp.,} right{\rm )}} perfect if $R$ is semiperfect and $J(R)$ is right (resp., left) t-nilpotent (see \cite{fu}). It is easy to see that left perfect property is a Morita invariant property.\\

Let $\Lambda$ be a left artinian ring. The subcategory $\mathscr{X}$ of $\Lambda$-mod is called {\it of bounded representation type} if there is a finite upper bound for the lengths of the indecomposable modules in $\mathscr{X}$ (see \cite[Sect. 4]{wis}).

\begin{The}\label{2.1}
Assume that $\Lambda$ is a left artinian ring and $\mathscr{X}$ is a contravariantly finite resolving subcategory of $\Lambda${\rm -mod}. Let $\lbrace V_{\alpha}~|~\alpha \in J\rbrace$ be a complete set of representative of the isomorphic classes of indecomposable modules in $\mathscr{X}$ and $T=\widehat{{\rm End}}_{\Lambda}(V)$, where $V=\bigoplus_{\alpha \in J}V_{\alpha}$.
\begin{itemize}
\item[$(1)$] The following statements are equivalent:
\begin{itemize}
\item[$({\rm i})$] $\mathscr{X}$ is of finite representation type and ${\rm res.dim}_{\mathscr{X}}\Lambda${\rm -mod} $\leq 2$.
\item[$({\rm ii})$] $\mathscr{X}$ is of bounded representation type and ${\rm res.dim}_{\mathscr{X}}\Lambda${\rm -mod} $\leq 2$.
\item[$({\rm iii})$] $T$ is a right perfect ring and ${\rm l.gl.dim}T \leq 2$.
\item[$({\rm iv})$] $T$ is a left locally finite ring and ${\rm l.gl.dim}T \leq 2$.
\end{itemize}
\item[$(2)$] The following statements are equivalent:
\begin{itemize}
\item[$({\rm i})$]  $T$ is a right locally noetherian ring and every module in $\mathscr{A}={\underrightarrow{\lim}}\mathscr{X}$ is a direct sum of finitely generated indecomposable modules.
\item[$({\rm ii})$] $T$ is a right locally finite ring.
\item[$({\rm iii})$] $V_T$ has finite length.
\end{itemize}
\item[$(3)$] If $T$ is a right locally finite ring and ${\rm l.gl.dim}T \leq 2$, then $\mathscr{X}$ is of finite representation type and ${\rm res.dim}_{\mathscr{X}}\Lambda${\rm -mod} $\leq 2$. Also the converse holds when ${\rm res.dim}_{\mathscr{X}}\Lambda${\rm -mod} $= 0$.
\end{itemize}
\end{The}
\begin{proof}
 By \cite[Proposition 52.5]{wi}, there is an equivalence $\textbf{F}: R{\rm Mod} \longrightarrow {\rm Mod}(\mathscr{X})$ which preserves direct sums, projective modules, finitely generated modules, exact sequences and flat modules, where $R$ is the functor ring of $\mathscr{X}$. \\
$(1)~({\rm i}) \Rightarrow ({\rm ii})$  is clear.\\
$({\rm ii}) \Rightarrow ({\rm iii})$.  Because $\mathscr{X}$ is of bounded representation type, by \cite[Propositions 54.1]{wi},  $T$ is a semiprimary ring, that is,  $T/J(T)$ is semisimple and $J(T)$ is nilpotent. This implies that $T$ is a left and right perfect ring. Now we show that ${\rm l.gl.dim}T \leq 2$. To do this we need to show that if $g: Q \rightarrow Q'$ is a $T$-module homomorphism with $Q, Q' \in {\rm Proj}(T)$, then ${\rm Ker}g \in {\rm Proj}(T)$. It is easy to checked that $Q=\bigoplus_{i\in A}Q_i$ and $Q'=\bigoplus_{j\in B}Q'_j$, where $A$ and $B$ are two sets and $Q_i$ and $Q'_j$ are finitely generated indecomposable projective unitary left $T$-modules, for each $i \in A, j \in B$. Assume that $E$ is a set of all finite subsets of $A$. It is not difficult to see that $(E, \subseteq)$ is a quasi-ordered directed set and $(Q_L, \ell_{LK})_{L,K \in E}$ is a direct system of modules, where each $Q_L=\bigoplus_{i \in L}Q_i$ and for each $L \subseteq K$, a morphism $\ell_{LK}: Q_L \rightarrow Q_K$ is the canonical injection. Moreover $Q$ together with the canonical injections $\ell_L: Q_L \rightarrow Q$ form a direct limit of $(Q_L, \ell_{LK})_{L,K \in E}$. Put $g_L:=g \circ \ell_L$ for each $L \in E$. Then $({\rm Ker}g_L, \ell_{LK})_{L,K \in E}$ is a direct system of modules and there is a monomorphism $\gamma: {\underrightarrow{\lim}}{\rm Ker}g_L \rightarrow Q$ such that the following diagram is commutative
\begin{displaymath}
\xymatrix{
{\rm Ker}g_L \ar@{->}[dd] \ar@/_/[dr]  \ar@{->}[rr]^{\ell_{LK}} && {\rm Ker}g_K \ar@{->}[dd] \ar@/^/[dl]  \\
&  {\underrightarrow{\lim}}{\rm Ker}g_L \ar@/^/[dd]^{\gamma}\\
Q_L \ar@{->}[rr] \ar@/_/[dr]_{\ell_L} \ar@/_/[ddr]_{g_L} \ar@{->}[rr]^{\ell_{LK}} && Q_K  \ar@/^/[dl]^{\ell_K} \ar@/^/[ddl]^{g_K} \\
&  Q \ar[d]_{g}\\
& Q'\\
 }
\end{displaymath}
We can see that $0 \rightarrow {\underrightarrow{\lim}}{\rm Ker}g_L \overset{\gamma}{\rightarrow} Q \overset{g}{\rightarrow} Q'$ is an exact sequence. This implies that ${\rm Ker}g \cong {\underrightarrow{\lim}}{\rm Ker}g_L$. If each ${\rm Ker}g_L$ is a finitely generated projective unitary left $T$-module, then  ${\rm Ker}g$ is flat. But since $T$ is left perfect, ${\rm Ker}g$ is projective. So it is enough to show that if $f: P \rightarrow P'$ is a $T$-module homomorphism with $P, P' \in {\rm proj}(T)$, then ${\rm Ker}f \in {\rm proj}(T)$. Since $\Lambda$ is left artinian, by Lemma \ref{1.5} there is an  additive equivalence $\mathbf{G}: T{\rm Mod} \rightarrow R{\rm Mod}$.  Then we have the exact sequence
 $$0 \rightarrow \mathbf{FG}({\rm Ker}f) \rightarrow \mathbf{FG}(P) \overset{\mathbf{FG}(f)}{\rightarrow} \mathbf{FG}(P') \rightarrow \mathbf{FG}({\rm Coker}f) \rightarrow 0.$$
 Since the functors $\mathbf{F}$ and $\mathbf{G}$ preserve finitely generated projective modules,  by \cite[Proposition 2.1]{ausla2}, we have the commutative diagrams
\begin{displaymath}
\xymatrix{
0 \ar[r] & \mathbf{FG}({\rm Ker}f) \ar[r] \ar[d]_{\cong}  &
\mathbf{FG}(P) \ar[r]^{\mathbf{FG}(f)}  \ar[d]_{\cong} & \mathbf{FG}(P') \ar[d]_{\cong} \\
0 \ar [r] & \mathscr{X}(-,{\rm Ker}h) \ar[r] &  \mathscr{X}(-,X) \ar[r]^{\mathscr{X}(-,h)} & \mathscr{X}(-,X') }
\end{displaymath}
where the below row is exact and $h: X \rightarrow X'$ is a morphism in $\mathscr{X}$.  Since ${\rm res.dim}_{\mathscr{X}}\Lambda${\rm -mod} $\leq 2$, by Lemma \ref{1.7}, ${\rm Ker}h \in \mathscr{X}$. This yields that ${\rm Ker}f$ is a finitely generated projective unitary left $T$-module. Therefore ${\rm l.gl.dim}T \leq 2$.\\
$({\rm iii}) \Rightarrow ({\rm iv})$. Since $\Lambda \in \mathscr{X}$, $V$ is a generator in $\Lambda$-Mod and hence by Lemma \ref{1.6}, $T$ is a left locally noetherian ring. Since $T$ is a right perfect ring,  $T$ satisfies the descending chain condition on its finitely generated left ideals. It follows that every finitely generated unitary left $T$-module is artinian. Therefore $T$ is a left locally finite ring.\\
$({\rm iv}) \Rightarrow ({\rm i})$. The functor $\widehat{{\Hom}}_{\Lambda}(V,-)$ establishes an equivalence between ${\rm Add}(V)$ and the full subcategory ${\rm Proj}(T)$ of $T$Mod. Hereby indecomposable modules in $\mathscr{X}$ correspond to finitely generated indecomposable projective unitary left $T$-modules which are local by \cite[Remark 2.3]{za}. By using \cite[Lemma 2.4]{za} we can see that the functor $\widehat{{\Hom}}_{\Lambda}(V,-)$ yields a bijection between the complete set of representative of the isomorphic classes of indecomposable modules in $\mathscr{X}$ and the set of
projective covers of non-isomorphic simple unitary left $T$-modules. Since $\Lambda$ is left artinian, there are only finitely many non-isomorphic simple left $\Lambda$-modules $S_1, \cdots, S_n$.  For every  indecomposable module $X$ in $\mathscr{X}$, there is an epimorphism $g: X \rightarrow S_j$ for some $1 \leq j \leq n$.  Since $V$ is a generator in $\Lambda$-Mod,  by \cite[Proposition 51.5(1)]{wi} $\widehat{{\Hom}}_{\Lambda}(V,g): \widehat{{\Hom}}_{\Lambda}(V,X) \rightarrow \widehat{{\Hom}}_{\Lambda}(V,S_j)$ is non-zero. Hence the simple factor module of $\widehat{{\Hom}}_{\Lambda}(V,X)$ occurs as a composition factor of $\widehat{{\Hom}}_{\Lambda}(V,S_j)$. On the other hand, since $\mathscr{X}$ is contravariantly finite and $\Lambda \in \mathscr{X}$, each $\widehat{{\Hom}}_{\Lambda}(V,S_i)$ is finitely generated and so it has finite length since $T$ is a left locally finite ring. Hence there are only finitely many non-isomorphic simple unitary left $T$-modules and  consequently there exists only a finite number of non-isomorphic indecomposable modules in $\mathscr{X}$. Therefore $\mathscr{X}$ is of finite representation type.

Now we show that ${\rm res.dim}_{\mathscr{X}}\Lambda${\rm -mod} $\leq 2$. To do this we need to show that for each finitely generated left $\Lambda$-module $M$, ${\rm res.dim}_{\mathscr{X}}M \leq 2$.  Let $M$ be a finitely generated left $\Lambda$-module. Since $\mathscr{X}$ is a contravariantly finite subcategory of $\Lambda$-mod, we have an exact sequence $\mathscr{X}(-,X) \overset{\mathscr{X}(-,f)}{\rightarrow} \mathscr{X}(-,M) \rightarrow 0$ with $X \in \mathscr{X}$. We know that by Lemma \ref{1.5}, $T$ is Morita equivalent to the functor ring $R$ of $\mathscr{X}$. Hence $\mathscr{X}(-,M) \cong \mathbf{F}\mathbf{G}(N)$ for some finitely generated left $T$-module $N$, where $\mathbf{G}: T${\rm -Mod} $\rightarrow R{\rm Mod}$ is an equivalence of categories. Now let $g: P_0 \rightarrow N$ be an epimorphism with $P_0 \in {\rm proj}(T)$. Then we have an exact sequence $\mathbf{F}\mathbf{G}(P_0) \rightarrow \mathscr{X}(-,M) \rightarrow 0$. But by \cite[Proposition 2.1]{ausla2}, $\mathbf{F}\mathbf{G}(P_0) \cong \mathscr{X}(-,X_0)$ for some $X_0 \in \mathscr{X}$. Then there is a morphism $\mathscr{X}(-,h): \mathscr{X}(-,X_0) \rightarrow \mathscr{X}(-,X)$ such that the following diagram is commutative
\begin{displaymath}
\xymatrix{
\mathbf{F}\mathbf{G}(P_0) \ar[r] & \mathscr{X}(-,M) \ar@{=}[d] \ar[r]& 0\\
 \mathscr{X}(-,X_0) \ar[u]^{\cong} \ar[r]^{\mathscr{X}(-,f\circ h)} & \mathscr{X}(-,M) \ar[r]& 0}
\end{displaymath}
Since ${\rm l.gl.dim}T \leq 2$ and $T$ is a left artinian ring, by the similar argument we have the following commutative diagram
\begin{displaymath}
\xymatrix{
0 \ar[r] & \mathbf{F}\mathbf{G}(P_2) \ar[r] \ar[d]_{\cong}  &
\mathbf{F}\mathbf{G}(P_1) \ar[r]  \ar[d]_{\cong} & \mathbf{F}\mathbf{G}(P_0) \ar[d]_{\cong} \ar[r] & \mathscr{X}(-,M) \ar@{=}[d] \ar[r] & 0\\
0 \ar [r] & \mathscr{X}(-, X_2) \ar[r]^{\mathscr{X}(-,g_2)} &  \mathscr{X}(-,X_1) \ar[r]^{\mathscr{X}(-,g_1)} & \mathscr{X}(-,X_0) \ar[r]^{\mathscr{X}(-,f\circ h)} & \mathscr{X}(-,M) \ar[r] & 0}
\end{displaymath}
where $X_i \in \mathscr{X}$, $P_i \in {\rm proj}(T)$ and the rows are exact. Therefore ${\rm res.dim}_{\mathscr{X}}M \leq 2$ and the result follows.\\

$(2) ~ ({\rm i}) \Rightarrow ({\rm ii})$. Since $\Lambda$-Mod is a locally finitely presented category with products and also ${\rm fp}(\Lambda$-Mod$)= \Lambda$-mod,  by  \cite[Theorem 4.1]{wc}, the full subcategory $\mathscr{A}={\underrightarrow{\lim}}\mathscr{X}$ of $\Lambda$-Mod is a locally finitely presented category. Also ${\rm fp}(\mathscr{A})=\mathscr{X}$ because $\mathscr{X}$ is closed under direct summands.  From \cite[Theorem 1.4(2)]{wc} we observe that the functor $\mathbf{H}:  \mathscr{A} \rightarrow {\rm Mod}(\mathscr{X})$ via $A \mapsto \mathscr{X}(-,A)$ induces an equivalence between the category $\mathscr{A}$ and the category ${\rm Flat}(\mathscr{X})$. Hence by assumption we can easily see that $T$ is a left perfect ring. This implies that $T$ satisfies the descending chain condition on its finitely generated right ideals. Since $T$ is right locally noetherian, $T$ is a right locally finite ring.\\
$({\rm ii}) \Rightarrow ({\rm iii})$. We know that $V$ is a generator in $\Lambda$-Mod because $\Lambda \in \mathscr{X}$. Hence $V_T$ is finitely generated projective right $T$-module and so by assumption $V_T$ has finite length.\\
$({\rm iii}) \Rightarrow ({\rm i})$. Since $V_T$ has finite length,  we can easily see that $T$ is a right locally noetherian ring. Now consider the descending chain of $T$-submodules of $V$
$$VJ(T) \supseteq V{J(T)}^2 \supseteq \ldots$$
Since $V$ is an artinian right $T$-module, $V{J(T)}^n=V{J(T)}^{n+1}$ for some $n \in {\Bbb{N}}$. By Nakayama lemma we conclude $V{J(T)}^n=0$ because $V$ is a noetherian right $T$-module. Since  $V_T$ is faithful,  ${J(T)}^n=0$. On the other hand, since $T=\bigoplus_{\alpha \in J}Te_{\alpha}$, by \cite[Lemma 2.4]{za}, $T/J(T)\cong \bigoplus_{\alpha \in J}Te_{\alpha}/{\rm rad}(Te_{\alpha})$ is a semisimple ring.  Therefore $T$ is a left and right perfect ring. Since the functor $\textbf{F}: R{\rm Mod} \longrightarrow {\rm Mod}(\mathscr{X})$ is an equivalence which preserves flat modules and finitely generated projective modules, for each $A \in \mathscr{A}$ we have ${\rm Hom}_{\Lambda}(-,A)|_{\mathscr{X}} \simeq {\rm Hom}_{\Lambda}(-,\bigoplus_{i\in I}X_i)|_{\mathscr{X}}$, where each $X_i \in \mathscr{X}$. Since $\Lambda \in \mathscr{X}$, $A \cong \bigoplus_{i\in I}X_i$. Hence every module in $\mathscr{A}$ is a direct sum of finitely generated indecomposable modules.\\

$(3)$. If $T$ is a right locally finite ring, then by the proof of $((iii) \Rightarrow (i))$ of part $(2)$, $T$ is right perfect and so by $(1)$, $\mathscr{X}$ is of finite representation type and ${\rm res.dim}_{\mathscr{X}}\Lambda${\rm -mod} $\leq 2$ when ${\rm l.gl.dim}T \leq 2$. Now assume that $\mathscr{X}$ is of finite representation type and ${\rm res.dim}_{\mathscr{X}}\Lambda${\rm -mod} $=0$, then $\Lambda$ is of finite representation type and so by \cite[Proposition 4.2]{ausla2}, $T$ is an artinian ring with ${\rm l.gl.dim}T \leq 2$.
\end{proof}

The following example shows that the converse of Theorem \ref{2.1}$(3)$ doesn't hold when ``$0 < {\rm res.dim}_{\mathscr{X}}\Lambda${\rm -mod} $\leq 2$".

\begin{Examp}{\rm
Let $G \subseteq F$ be fields such that ${\rm dim}~ F_G= \infty$ and set $R=\begin{pmatrix}
F & F\\
0& G
\end{pmatrix}$ and $S=FQ/I$, where $Q$ is the quiver
$$\hskip .5cm \xymatrix@-1mm{
{1} \ar@{<-}[r]^{\alpha} &{2} \ar@{<-}[r]^{\beta} & {3}}$$ and  $I$ is the ideal of $FQ$ generated by $\beta \alpha$.\\
$({\rm i})$ By \cite[Proposition 1.1]{sim}, $R$ is a left hereditary left artinian ring but it isn't a right artinian ring. Set $\mathscr{Y}:={\rm proj}(R)$. It is easy to see that $\mathscr{Y}$ is a contravariantly finite resolving subcategory of $R$-mod with ${\rm res.dim}_{\mathscr{Y}}R${\rm -mod} $= 1$. Moreover it is of finite representation type while the Auslander ring $T$ of $\mathscr{Y}$ isn't a right artinian ring, because $R$ is Morita equivalent to $T$.\\
$({\rm ii})$ We know that $\Gamma=R \oplus S^{\rm op}$ is a left artinian ring. But it isn't a right artinian ring. Put $\mathscr{Z} :={\rm proj}(\Gamma)$. We can see that $\mathscr{Z}$ is a contravariantly finite resolving subcategory of $\Gamma$-mod with  ${\rm res.dim}_{\mathscr{Z}}\Gamma${\rm -mod} $= 2$. In particular $\mathscr{Z}$ is of finite representation type. But because the Auslander ring $T'$ of $\mathscr{Z}$ is Morita equivalent to $\Gamma$, the Auslander ring $T'$ of $\mathscr{Z}$ isn't a right artinian ring.}
\end{Examp}

Let $F$ be an $\mathscr{X}^{\rm op}$-module and $X \in \mathscr{X}$. An element $x \in F(X)$ is said to be a {\it minimal element} if $x \neq 0$ and for any proper epimorphism $f: X \rightarrow X'$ with $X' \in \mathscr{X}$, $F(f)(x) = 0$ (see \cite[Page 292]{ausla2}).

The following lemmas are analogue of \cite[Lemma 3.2 and Propositions 3.3 and 3.4]{ausla2}.

\begin{Lem}\label{1.1}
Let $F$ be an $\mathscr{X}^{\rm op}$-module and $X \in \mathscr{X}$.
\begin{itemize}
\item[$({\rm i})$] If $F(X)$ has a minimal element, then $X$ is an indecomposable left $\Lambda$-module.
\item[$({\rm ii})$] If $X$ is a noetherian left $\Lambda$-module, then for each $0 \neq x \in F(X)$ there is an epimorphism $f: X \rightarrow X'$ with $X' \in \mathscr{X}$ such that $F(f)(x)$ is a minimal element in $F(X')$.
\end{itemize}
\end{Lem}
\begin{proof}
$({\rm i})$. It follows from the fact that $\mathscr{X}$ is closed under direct summands.\\
$({\rm ii})$. Let $0 \neq x \in F(X)$ and $\mathcal{S}$ be the set of all $\Lambda$-submodules $Y$ of $X$ such that $X/Y \in \mathscr{X}$ and the canonical projection $\pi: X \rightarrow X/Y$ has the property that $F(\pi)(x)\neq 0$. Since $\mathcal{S}$ is not empty, there is a maximal element $Y'$ in $\mathcal{S}$.  Then $X/Y'\in \mathscr{X}$ and $F(\pi)(x)\neq 0$, where $\pi: X \rightarrow X/Y'$ is the canonical projection. We can see that $x'=F(\pi)(x)$ is a minimal element in $F(X/Y')$ by applying the technique used in the proof of \cite[Lemma 3.2(b)]{ausla2}. Therefore the result holds.
\end{proof}

We recall from \cite{ausla2} that a family of morphisms is called {\it noetherian} if for each sequence $X_1 \overset{f_1}{\rightarrow} X_2 \overset{f_2}{\rightarrow} X_3 \rightarrow \cdots$ of morphisms in the family such that $f_i \circ f_{i-1} \circ \cdots \circ f_1 \neq 0$ for all $i$, there is an integer $n$ such that $f_k$ is an isomorphism for all $k \geq n$. Also a family of morphisms is said to be {\it artinian} if for each sequence $ \cdots \rightarrow X_3 \overset{g_2}{\rightarrow} X_2 \overset{g_1}{\rightarrow} X_1$ of morphisms in the family such that $g_1 \circ g_{2} \circ \cdots \circ g_i \neq 0$ for all $i$, there is an integer $n$ such that $g_k$ is an isomorphism for all $k \geq n$ (see \cite[Page 290]{ausla2}). Let $F$ be an $\mathscr{X}^{\rm op}$-module and $X \in \mathscr{X}$. An element $x \in F(X)$ is said to be a {\it universally minimal element} if $x \neq 0$ and for any morphism $f: X \rightarrow X'$ in $\mathscr{X}$ which isn't an split monomorphism, $F(f)(x) = 0$ (see \cite[Page 292]{ausla2}).

\begin{Lem}\label{1.2}
Let $\Lambda$ be a left noetherian ring and $F$ be a non-zero $\mathscr{X}^{\rm op}$-module. If every family of morphisms between indecomposable modules in $\mathscr{X}$ is noetherian, then $F(X)$ has a universally minimal element for some $X \in \mathscr{X}$.
\end{Lem}
\begin{proof}
By using Lemma \ref{1.1}, the assumption and by applying the technique used in the proof of \cite[Proposition 3.3]{ausla2}, we can see that there is an $X \in \mathscr{X}$ and a minimal element $x \in F(X)$ with the property that any morphism $f: X \rightarrow X'$ in $\mathscr{X}$ is an isomorphism if $F(f)(x)$ is a minimal element in $F(X')$ and then the result follows.
\end{proof}

Let $\Lambda$ be a left artinian ring. We recall from \cite{ausla2} that a morphism $f:C \rightarrow D$ in $\mathscr{X}$ is called {\it right almost split} if
\begin{center}
\hspace{-9cm}$i)$ $f$ is not a split epimorphism;\\
\hspace{-0.5cm}$ii)$ for any morphism $g: X \rightarrow D$ in $\mathscr{X}$ which is not a split epimorphism, there is a\\
 \hspace{-7.5cm}morphism $h: X \rightarrow C$ such that $f\circ h = g$.
\end{center}

\begin{Lem}\label{1.3}
Let $\Lambda$ be a left artinian ring and $\mathscr{X}$ be a resolving subcategory of $\Lambda$-{\rm mod} with the property that every family of morphisms between indecomposable modules in $\mathscr{X}$ is noetherian. Then for each indecomposable non-projective left $\Lambda$-module $X \in \mathscr{X}$, there is a right almost split morphism $f: B \rightarrow X$ in $\mathscr{X}$.
\end{Lem}
\begin{proof}
Let $X$ be an indecomposable non-projective left $\Lambda$-module in $\mathscr{X}$. Since $\mathscr{X}$ is resolving we can see that the covariant additive functor ${\rm Ext}_{\Lambda}^1(X,-): \mathscr{X} \rightarrow \mathit{Ab}$ is non-zero. Hence by Lemma \ref{1.2}, ${\rm Ext}_{\Lambda}^1(X,A)$ has a universally minimal element for some $A \in \mathscr{X}$. Let $x:= \hspace{3mm} 0 \rightarrow A \rightarrow B \overset{g}{\rightarrow} X \rightarrow 0$  be a universally minimal element in  ${\rm Ext}_{\Lambda}^1(X,A)$. Then $g$ is not a split epimorphism. Since $\mathscr{X}$ is resolving and $\Lambda$ is a left artinian ring,  we get $g: B \rightarrow X$ is a right almost split morphism by applying the technique used in the proof of \cite[Proposition 3.4]{ausla2}.
\end{proof}

An  $\mathscr{X}$-module $M$ is called {\it noetherian {\rm (resp.,} artinian{\rm )}} if it satisfies the ascending (resp., descending) chain condition on its submodules. Also $M$ is said to be {\it finite} if it is both noetherian and artinian. An $\mathscr{X}$-module $M$ is called {\it locally finite} if every finitely generated submodule of $M$ is finite. We recall from \cite{ausla2} that the category ${\rm Mod}(\mathscr{X})$ is called {\it locally finite} if every $\mathscr{X}$-module is locally finite.\\

Now, we prove the following result that gives a characterization of contravariantly finite resolving subcategories $\mathscr{X}$ of $\Lambda$-mod of finite representation type in terms of Mod$(\mathscr{X})$.

\begin{The}\label{2.2}
Let $\Lambda$ be a left artinian ring and $\mathscr{X}$ be a contravariantly finite resolving subcategory of $\Lambda$-{\rm mod} which contains $J(\Lambda)$. Then the following statements are equivalent.
\begin{itemize}
\item[$({\rm i})$] $\mathscr{X}$ is of finite representation type.
\item[$({\rm ii})$] Any family of homomorphisms between indecomposable modules in $\mathscr{X}$ is both noetherian and artinian.
\item[$({\rm iii})$] ${\rm Mod}(\mathscr{X})$ is locally finite.
\item[$({\rm iv})$] $\mathscr{X}(-,X)$ is finite for each $X \in \mathscr{X}$.
\item[$({\rm v})$] $\mathscr{X}(-,S)$ is finite for each simple left $\Lambda$-module $S$.
\end{itemize}
\end{The}
\begin{proof}
$({\rm i}) \Rightarrow ({\rm ii})$. It follows from Harada-Sai Lemma.\\
$({\rm ii}) \Rightarrow ({\rm iii})$.  By \cite[Proposition 1.11]{ausla2}, it is enough to show that each simple $\mathscr{X}$-module is finitely presented and each non-zero $\mathscr{X}$-module has a simple submodule. Since any family of homomorphisms between indecomposable modules in $\mathscr{X}$ is noetherian, by Lemma \ref{1.3}, every indecomposable non-projective module in $\mathscr{X}$ has a right almost split morphism. Since $J(\Lambda) \in \mathscr{X}$,  every indecomposable projective module in $\mathscr{X}$ has a right almost split morphism. Therefore by \cite[Proposition 2.7]{ausla2}, each simple $\mathscr{X}$-module is finitely presented. Now we show that each non-zero $\mathscr{X}$-module has a simple submodule. We apply the technique used in the proof of $({\rm ii}) \Rightarrow ({\rm iii})$ of \cite[Theorem 3.1]{ab}. Let $F$ be a non-zero $\mathscr{X}$-module which has not simple submodule. There is an indecomposable module $Y_1$ in $\mathscr{X}$ such that $F(Y_1) \neq 0$. By Yoneda Lemma, there is a non-zero morphism $\alpha_0: \mathscr{X}(-,Y_1) \rightarrow F$. Set $F_0={\rm Im}\alpha_0$.  $F_0$ is a non-zero submodule of $F$. Since $F$ has not simple submodule, $F_0$ contains a proper non-zero submodule $F_1$. The above argument shows that there is a non-zero morphism $\alpha_1: \mathscr{X}(-,Y_2) \rightarrow F_1$, where $Y_2$ is an indecomposable module in $\mathscr{X}$.  Since $\mathscr{X}(-,Y_2)$ is projective, $\alpha_1 \neq 0$ and $F_1$ is a proper submodule of $F_0$, we have a non-zero non-isomorphism $g_1: Y_2 \rightarrow Y_1$. By the similar argument, we can see that $F_2={\rm Im}\alpha_1$ contains a proper non-zero submodule $F_3$. Also there is a non-zero morphism $\alpha_2: \mathscr{X}(-,Y_3) \rightarrow F_3$, where $Y_3$ is an indecomposable module in $\mathscr{X}$. This implies that there is a non-zero non-isomorphism $g_2: Y_3 \rightarrow Y_2$ that $g_1 \circ g_2 \neq 0$. By continuing this process, we obtain a sequence $\cdots \rightarrow Y_3 \overset{g_2}{\rightarrow} Y_2 \overset{g_1}{\rightarrow} Y_1$
of non-isomorphisms between indecomposable modules of $\mathscr{X}$ such that $g_1\circ g_2 \circ \cdots \circ g_n \neq 0$ for each $n \in {\Bbb{N}}$,  which is a contradiction. Therefore each non-zero $\mathscr{X}$-module has a simple submodule and the result follows. \\
$({\rm iii}) \Rightarrow ({\rm iv})$ is clear.\\
$({\rm iv}) \Rightarrow ({\rm v})$. It follows from the fact that every simple left $\Lambda$-module has a right $\mathscr{X}$-approximation.\\
$({\rm v}) \Rightarrow ({\rm i})$. One can prove this implication by the technique used in the proof of $(c) \Rightarrow (d)$ of \cite[Proposition 3.1]{ausla2}.\\
\end{proof}

Let $\lbrace U_{\alpha}~|~ \alpha \in J \rbrace$ be a complete set of representative of the isomorphic classes of finitely presented left $\Lambda$-modules and set $U= \bigoplus _{\alpha \in J} U_{\alpha}$. We recall from \cite[Sect. 1]{wis} that a short exact sequence in $\Lambda$-Mod is called {\it pure} if it remains exact under the functor $\widehat{{\rm Hom}}_{\Lambda}(U,-)$.  Equivalently it remains exact under the functor $M\otimes_{\Lambda} -$ for each finitely presented right $\Lambda$-module $M$ (see \cite[Proposition 34.5]{wi}). A ring $\Lambda$ is called {\it left} {\rm (resp., {\it right})} {\it pure semisimple} if every short pure exact sequence in $\Lambda$-Mod (resp., Mod-$\Lambda$) splits (see \cite[Sect. 3]{wis}).  A ring $\Lambda$ is called {\it of bounded representation type} if it is left artinian and $\Lambda$-mod is of bounded representation type (see \cite[Sect. 3]{wis}). Note that a ring $R$ with enough idempotents is called {\it left functor ring} of $\Lambda$ if $R$ is the functor ring of $\Lambda$-mod (see \cite[Sect. 2]{wis}).

\begin{Cor}{\rm (See \cite[Theorem 3.1]{ausla2} and \cite[Theorem 3.1]{wis})}
For a ring $\Lambda$ with left functor ring $R$, the following statements are equivalent.
\begin{itemize}
\item[$({\rm i})$] $\Lambda$ is of finite representation type.
\item[$({\rm ii})$] $\Lambda$ is of bounded representation type.
\item[$({\rm iii})$] $R$ is a left and right perfect ring.
\item[$({\rm iv})$] $R$ is a left locally finite ring.
\item[$({\rm v})$] $R$ is a right locally finite ring.
\item[$({\rm vi})$] $\Lambda$ is a left and right pure semisimple ring.
\item[$({\rm vii})$] $\Lambda$ is a left artinian ring and any family of homomorphisms between finitely generated indecomposable left $\Lambda$-modules is both noetherian and artinian.
\item[$({\rm viii})$] $\Lambda$ is a left artinian ring and ${\rm Mod}(\Lambda$-{\rm mod}$)$ is locally finite.
\item[$({\rm ix})$] $\Lambda$ is a left artinian ring and ${\rm Hom}_{\Lambda}(-,X)$ is finite for each finitely generated left $\Lambda$-module $X$.
\item[$({\rm x})$] $\Lambda$ is a left artinian ring and ${\rm Hom}_{\Lambda}(-,S)$ is finite for each simple left $\Lambda$-module $S$.
\end{itemize}
 \end{Cor}
 \begin{proof}
$({\rm i}) \Rightarrow ({\rm ii})$. Follows by Theorem \ref{2.1}(1).\\
$({\rm ii}) \Rightarrow ({\rm iii})$. Follows by Lemma \ref{1.5} and the proof of ($({\rm ii}) \Rightarrow ({\rm iii})$) of Theorem \ref{2.1}(1).\\
$({\rm iii}) \Rightarrow ({\rm iv})$. By \cite[Theorem]{fu} and \cite[Theorem 2]{hu}, $\Lambda$ is a left artinian ring and ${\rm l.gl.dim}R \leq 2$. Hence by Lemma \ref{1.5} and Theorem \ref{2.1}(1), the result follows.\\
$({\rm iv}) \Rightarrow ({\rm i})$. Since $R$ is a left locally finite, for each idempotent $e \in R$, $Re$ has finitely length. This implies that by \cite[Theorem 2.4]{f}, $\Lambda$ is a left artinian ring. On the other hand, by using \cite[Proposition 1.5]{f} and \cite[Proposition 50.4(2)]{wi}, we can see that ${\rm l. gl.dim}R \leq 2$. Therefore by Lemma \ref{1.5} and Theorem \ref{2.1}(1), the result follows. \\
$({\rm i}) \Rightarrow ({\rm vi})$. By $({\rm iii})$, $R$ is left perfect and so by \cite[Theorem]{fu}, $\Lambda$ is left pure semisimple. But by Lemma \ref{1.5} and Theorem \ref{2.1}(3), we can see that $R$ is a right locally noetherian ring. Therefore by \cite[Proposition 53.7 ]{wi}, $\Lambda$ is a right pure semisimple ring.\\
$({\rm vi}) \Rightarrow ({\rm v})$. It follows from \cite[Proposition 53.7]{wi}, Lemma \ref{1.5} and Theorem \ref{2.1}(2).\\
$({\rm v})\Rightarrow ({\rm iii})$. By using \cite[Propositions 31.11(4), 35.7, 51.8(2) and 52.1(6)]{wi}, we can see that $\Lambda$ is a left artinian ring. Hence by Lemma \ref{1.5} and the proof of ($({\rm ii}) \Rightarrow ({\rm i})$) of Theorem \ref{2.1}(2), the result follows.\\
$({\rm i}) \Rightarrow ({\rm vii}) \Rightarrow ({\rm viii}) \Rightarrow ({\rm ix}) \Rightarrow ({\rm x}) \Rightarrow ({\rm i})$ follows from Theorem \ref{2.2}.
 \end{proof}

\section{locally finitely presented pure semisimple categories}

In this section, we give a characterization of covariantly finite subcategories $\mathscr{X}$ of $\Lambda$-mod of finite representation type in terms of decomposition properties of $\mathscr{A}={\underrightarrow{\lim}}\mathscr{X}$.  Moreover we generalize and unify the results of Auslander \cite[Proposition 2.1]{aus}, Beligiannis \cite[Theorems 3.1, 4.10 and 4.11]{ab} and Chen \cite[Main Theorem]{che}. \\

Let $\mathscr{A}$ be a locally finitely presented category. A sequence $0 \rightarrow A \overset{f}{\rightarrow} B \overset{g}{\rightarrow} C \rightarrow 0$ in $\mathscr{A}$ (i.e., a pair of maps $f$ and $g$ with $g\circ f=0$) is said to be {\it pure-exact} if the sequence $$0 \rightarrow {\rm Hom}_{\mathscr{A}}(X, A) \overset{{\rm Hom}_{\mathscr{A}}(X, f)}{\rightarrow} {\rm Hom}_{\mathscr{A}}(X, B) \overset{{\rm Hom}_{\mathscr{A}}(X, g)}{\rightarrow} {\rm Hom}_{\mathscr{A}}(X, C) \rightarrow 0$$ is exact for each $X \in {\rm fp}(\mathscr{A})$. The category $\mathscr{A}$ is called {\it pure semisimple} if any pure-exact sequence in $\mathscr{A}$ splits (see \cite[Section 3]{wc}). Note that if $\mathscr{A}$ has products, then $\mathscr{A}$ is pure semisimple if and only if $\mathscr{A}={\rm Add}({\rm fp}(\mathscr{A}))$ (see \cite[Subsection 2.1.1]{ab}).

\begin{The}\label{3.1}
Let $\Lambda$ be a left artinian ring and $\mathscr{X}$ be a covariantly finite subcategory of $\Lambda$-mod. Assume that every simple $\mathscr{X}^{\rm op}$-module is finitely presented and $\mathscr{A}={\underrightarrow{\lim}}\mathscr{X}$. Then the following statements are equivalent.
\begin{itemize}
\item[$({\rm i})$] $\mathscr{X}$ is of finite representation type.
\item[$({\rm ii})$] Every module in $\mathscr{A}$ is a direct sum of finitely generated modules.
\item[$({\rm iii})$] $\mathscr{A}$ is pure semisimple.
\item[$({\rm iv})$] Any family of homomorphisms between indecomposable modules in $\mathscr{X}$ is noetherian.
\item[$({\rm v})$] ${\rm Mod}(\mathscr{X}^{\rm op})$ is locally finite.
\item[$({\rm vi})$] $\mathscr{X}(X,-)$ is finite for each $X \in \mathscr{X}$.
\item[$({\rm vii})$] $\mathscr{X}(S,-)$ is finite for each simple left $\Lambda$-module $S$.
\item[$({\rm viii})$] There is an additive equivalence between $\mathscr{A}$ and the category of projective modules over a semiperfect ring.
\item[$({\rm ix})$] There is an additive equivalence between $\mathscr{X}$ and the category of finitely generated projective modules over a semiperfect ring.
\end{itemize}
If $\Lambda \in \mathscr{X}$, then $({\rm i})$-$({\rm ix})$ are equivalent to
\begin{itemize}
\item[$({\rm x})$] Every module in $\mathscr{A}$ is a direct sum of indecomposable modules.
\end{itemize}
\end{The}
\begin{proof}
By the proof of $({\rm i}) \Rightarrow ({\rm ii})$ of Theorem \ref{2.1}(2), $\mathscr{A}$ is a locally presented category with ${\rm fp}(\mathscr{A})=\mathscr{X}$ and the functor $\mathbf{H}:  \mathscr{A} \rightarrow {\rm Mod}(\mathscr{X})$ via $A \mapsto \mathscr{X}(-,A)$ induces an equivalence between the category $\mathscr{A}$ and the category ${\rm Flat}(\mathscr{X})$.\\
$({\rm i}) \Rightarrow ({\rm ii})$.  Let $\lbrace V_1, \cdots, V_n\rbrace$ be a complete set of representative of the isomorphic classes of indecomposable modules in $\mathscr{X}$ and $T={\rm End}_{\Lambda}(V)$, where $V=\bigoplus_{i=1}^nV_i$. Since $V$ is finitely generated, by \cite[Proposition 32.4]{wi}, $T$ is a semiprimary ring and so it is a left perfect ring. We know that by Lemma \ref{1.5}, $T$ is Morita equivalent to the functor ring $R$ of $\mathscr{X}$. Hence by \cite[Theorem 3.4]{zn}, there is an equivalence $\mathbf{G}: {\rm Proj}(T) \rightarrow {\rm Proj}(R)$ which preserves and reflects direct sums and finitely generated modules. Also $R$ is a left perfect ring. This implies that every direct limit of finitely generated projective unitary left $R$-modules is projective. So by \cite[Proposition 52.5]{wi}, there is an equivalence $\textbf{F}: {\rm Proj}(R) \longrightarrow {\rm Flat}(\mathscr{X})$ which preserves direct sums and finitely generated modules. It is easy to check that the composition of the following functors is an equivalence which preserves direct sums and finitely generated modules
$${\rm Add}(V) \overset{{\rm Hom}_{\Lambda}(V,-)}{\rightarrow} {\rm Proj}(T) \overset{\mathbf{G}}{\rightarrow} {\rm Proj}(R) \overset{\mathbf{F}}{\rightarrow} {\rm Flat}(\mathscr{X}) \overset{\mathbf{H}^{-1}}{\rightarrow} \mathscr{A}$$
Therefore every module in $\mathscr{A}$ is a direct sum of finitely generated indecomposable modules. \\
$({\rm ii}) \Rightarrow ({\rm iii})$. Since $\mathscr{X}$ is covariantly finite, by \cite[Theorem 4.2]{wc},  $\mathscr{A}$ is a locally finitely presented category with products. Since ${\rm Add}(\mathscr{X}) \subseteq \mathscr{A}$, it is enough to show that $\mathscr{A} \subseteq {\rm Add}(\mathscr{X})$. Assume that $M$ be a module in $\mathscr{A}$. Then $M=\bigoplus_{i\in I}M_i$, where $I$ is a set and each $M_i$ is a finitely generated left $\Lambda$-module. Since by \cite[Lemma 3.10]{sp}, $\mathscr{A}$ is closed under direct summands,  each $M_i \in \mathscr{A}$. But since by the proof of \cite[Theorem 4.1]{wc}, ${\rm Hom}_{\Lambda}(M_i, -): \mathscr{A} \rightarrow \mathfrak{Ab}$ preserves direct limits, each $M_i \in \mathscr{X}$. Therefore $M \in {\rm Add}(\mathscr{X})$ and the result follows. \\
$({\rm iii}) \Rightarrow ({\rm iv})$. It follows from \cite[Theorem 3.2]{wc}. \\
$({\rm iv}) \Rightarrow ({\rm v})$. By the similar argument in the proof of $({\rm ii})\Rightarrow ({\rm iii})$ of Theorem \ref{2.2}, we can see that every non-zero $\mathscr{X}^{\rm op}$-module has a simple submodule. The assumption yields ${\rm Mod}(\mathscr{X}^{\rm op})$ is locally finite.\\
$({\rm v}) \Rightarrow ({\rm vi})$ is clear.\\
$({\rm vi}) \Rightarrow ({\rm vii})$. It follows from the fact that every simple left $\Lambda$-module has a left $\mathscr{X}$-approximation.\\
$({\rm vii}) \Rightarrow ({\rm i})$. One can prove this implication by the technique used in the proof of $(c) \Rightarrow (d)$ of \cite[Proposition 3.1]{ausla2}.\\
$({\rm i}) \Leftrightarrow ({\rm viii}) \Leftrightarrow ({\rm ix})$. It follows from the fact that the covariant functor $\widehat{{\Hom}}_{\Lambda}(V,-)$ is an additive equivalence between the full subcategory ${\rm Add}(V)$ of $\Lambda$-Mod and the full subcategory Proj$(T)$ of $T$Mod, where $T=\widehat{{\rm End}}_{\Lambda}(V)$, $V=\bigoplus_{\alpha \in J}V_{\alpha}$ and $\lbrace V_{\alpha}~|~\alpha \in J\rbrace$ is a complete set of representative of the isomorphic classes of indecomposable modules in $\mathscr{X}$. \\
$({\rm i}) \Rightarrow ({\rm x})$ is clear.\\
$({\rm x}) \Rightarrow ({\rm i})$. Now assume that $\Lambda \in \mathscr{X}$ and every module in $\mathscr{A}$ is a direct sum of indecomposable modules. We show that $\mathscr{A}$ is pure semisimple. It is enough to show that every pure-exact sequence in $\mathscr{A}$ splits. Consider the pure-exact sequence $0 \rightarrow A \overset{f}{\rightarrow} B \overset{g}{\rightarrow} C \rightarrow 0$ in $\mathscr{A}$. Then $$0 \rightarrow {\rm Hom}_{\Lambda}(X,A) \overset{{\rm Hom}_{\Lambda}(X,f)}{\rightarrow} {\rm Hom}_{\Lambda}(X,B) \overset{{\rm Hom}_{\Lambda}(X,g)}{\rightarrow} {\rm Hom}_{\Lambda}(X,C) \rightarrow 0$$ is exact for each $X \in \mathscr{X}$. Hence the sequence $0 \rightarrow A \overset{f}{\rightarrow} B \overset{g}{\rightarrow} C \rightarrow 0$ is a short exact sequence because $\Lambda \in \mathscr{X}$.  But since every module in $\mathscr{A}$ is a direct sum of indecomposable modules and $\mathscr{X}$ is a covariantly finite subcategory of $\Lambda$-mod, by \cite[Corollary 2.7 and Proposition 3.11]{sp}, it is enough to show that the short exact sequence $0 \rightarrow A \overset{f}{\rightarrow} B \overset{g}{\rightarrow} C \rightarrow 0$ is a pure exact sequence in $\Lambda$-Mod. Let $M$ be a finitely generated left $\Lambda$-module and $\alpha: M \rightarrow C$ be a homomorphism. Since $C \in \mathscr{A}$ and $\Lambda$ is a left artinian ring, by \cite[Lemma 4.1]{wc}, $\alpha$ factors through a module in $\mathscr{X}$. There are homomorphisms $\beta: M \rightarrow Y$ and $\gamma: Y \rightarrow C$ such that $\alpha = \gamma \circ \beta$, where $Y \in \mathscr{X}$. Since  $$0 \rightarrow {\rm Hom}_{\Lambda}(Y,A) \overset{{\rm Hom}_{\Lambda}(Y,f)}{\rightarrow} {\rm Hom}_{\Lambda}(Y,B) \overset{{\rm Hom}_{\Lambda}(Y,g)}{\rightarrow} {\rm Hom}_{\Lambda}(Y,C) \rightarrow 0$$ is exact, $\gamma$ factors through $g$ and so $\alpha$ factors through $g$. Therefore the sequence $0 \rightarrow A \overset{f}{\rightarrow} B \overset{g}{\rightarrow} C \rightarrow 0$ is a pure exact sequence in $\Lambda$-Mod and the result follows.
\end{proof}

\begin{Cor}{\rm (See \cite[Proposition 2.1]{aus})}
Let $\Lambda$ be a left artinian ring and assume that every simple $(\Lambda$-{\rm mod}$)^{\rm op}$-module is finitely presented. Then $\Lambda$ is of finite representation type if and only if every left $\Lambda$-module is a direct sum of finitely generated left $\Lambda$-modules.
\end{Cor}

Let $\Lambda$ be an artin algebra and assume that $\mathscr{X}$ is a functorially finite subcategory of $\Lambda$-mod.  From  \cite[Theorem 2.3]{au} and \cite[Proposition 3.2]{auu} we can see that all simple functors in Mod$(\mathscr{X})$ and Mod$(\mathscr{X}^{\rm {op}})$ are finitely presented. So as a consequence of Theorem \ref{3.1} we have the following result.

\begin{Cor}\label{3.2}
Let $\Lambda$ be an artin algebra,  $\mathscr{X}$ be a functorially finite subcategory of $\Lambda$-mod and let $\mathscr{A}={\underrightarrow{\lim}}\mathscr{X}$. Then the following statements are equivalent.
\begin{itemize}
\item[$({\rm i})$] $\mathscr{X}$ is of finite representation type.
\item[$({\rm ii})$] Every module in $\mathscr{A}$ is a direct sum of finitely generated modules.
\item[$({\rm iii})$] $\mathscr{A}$ is pure semisimple.
\item[$({\rm iv})$] Any family of homomorphisms between indecomposable modules in $\mathscr{X}$ is noetherian.
\item[$({\rm v})$] ${\rm Mod}(\mathscr{X})$ is locally finite.
\item[$({\rm vi})$] $\mathscr{X}(X,-)$ is finite for each $X \in \mathscr{X}$.
\item[$({\rm vii})$] $\mathscr{X}(S,-)$ is finite for each simple left $\Lambda$-module $S$.
\item[$({\rm viii})$] There is an additive equivalence between $\mathscr{A}$ and the category of projective modules over an artin algebra.
\item[$({\rm ix})$] There is an additive equivalence between $\mathscr{X}$ and the category of finitely generated projective modules over an artin algebra.
\end{itemize}
If $\Lambda \in \mathscr{X}$, then $({\rm i})$-$({\rm ix})$ are equivalent to
\begin{itemize}
\item[$({\rm x})$] Every module in $\mathscr{A}$ is a direct sum of indecomposable modules.
\end{itemize}
\end{Cor}

Krause and Solberg in \cite[Corollary 2.6]{ks} showed that if $\Lambda$ is an artin algebra and $\mathscr{X}$ is a contravariantly finitely resolving subcategory of $\Lambda$-mod, then $\mathscr{X}$ is covariantly finite. Hence as a consequence of Theorem \ref{3.1} we have the following result.

\begin{Cor}{\rm (See \cite[Theorem 3.1]{ab})}
Let $\Lambda$ be an artin algebra, $\mathscr{X}$ be a contravariantly finite resolving subcategory of $\Lambda$-{\rm mod} and $\mathscr{A}={\underrightarrow{\lim}}\mathscr{X}$. Then the following statements are equivalent.
\begin{itemize}
\item[$({\rm i})$] $\mathscr{X}$ is of finite representation type.
\item[$({\rm ii})$] Every module in $\mathscr{A}$ is a direct sum of finitely generated modules.
\item[$({\rm iii})$] Every module in $\mathscr{A}$ is a direct sum of indecomposable modules.
\item[$({\rm iv})$] ${\rm Mod}(\mathscr{X})$ is locally finite.
\item[$({\rm v})$] There is an additive equivalence between $\mathscr{A}$ and the category of projective modules over an artin algebra.
\item[$({\rm vi})$] There is an additive equivalence between $\mathscr{X}$ and the category of finitely generated  projective modules over an artin algebra.
\end{itemize}
\end{Cor}

\subsection{Gorenstein projective modules}
An exact sequence of projective (resp., injective) left $\Lambda$-modules $\mathbf{X}^\bullet:= \cdots \rightarrow X^{i-1} \rightarrow X^i \rightarrow X^{i+1} \rightarrow \cdots$  is called {\it totally-acyclic} if the Hom-complex ${\rm Hom}_{\Lambda}(\mathbf{X}^\bullet, P)$ (resp., ${\rm Hom}_{\Lambda}(I, \mathbf{X}^\bullet)$) is exact for each $P \in {\rm Proj}(\Lambda)$ (resp., $I \in {\rm Inj}(\Lambda)$). A left $\Lambda$-module $M$ is called {\it Gorenstein projective} (resp., {\it Gorenstein injective}) if there is a totally-acyclic complex $\mathbf{X}^\bullet$ of projective (resp., injective) modules such that ${\rm Coker}(X^{-1} \rightarrow X^0) \cong M$ (resp., ${\rm Ker}(X^0 \rightarrow X^1) \cong M$).  Moreover we denote by ${\rm GProj}(\Lambda)$ (resp., ${\rm GInj}(\Lambda)$) the full subcategory of $\Lambda$-Mod consisting of all the Gorenstein projective (resp., Gorenstein injective) left $\Lambda$-modules. A left $\Lambda$-module $M$ is called {\it finitely generated Gorenstein projective} if it is finitely generated and Gorenstein projective and we denote by  ${\rm Gproj}(\Lambda)$ the full subcategory of $\Lambda$-mod consisting of all finitely generated Gorenstein projective left $\Lambda$-modules (see \cite[Definition 2.1]{ho}).  Set ${({\rm GProj}(\Lambda))}^\bot= \lbrace X \in \Lambda${\rm -Mod}$~|~ {\rm Ext}_{\Lambda}^n(M,X)=0 ~\forall n \in {\Bbb{N}}, ~ \forall M \in {\rm GProj}(\Lambda)\rbrace$ and $^\bot{({\rm GInj}(\Lambda))}= \lbrace X \in \Lambda${\rm -Mod}$~|~ {\rm Ext}_{\Lambda}^n(X,M)=0 ~\forall n \in {\Bbb{N}}, ~ \forall M \in {\rm GInj}(\Lambda)\rbrace$. We recall from \cite{ab} that an artin algebra $\Lambda$ is called {\it virtually Gorenstein} if ${({\rm GProj}(\Lambda))}^\bot = {^\bot{({\rm GInj}(\Lambda))}}$.\\

Let $\Lambda$ be an artin algebra. Then by \cite[Theorem 2.5]{ho}, the full subcategory ${\rm Gproj}(\Lambda)$ of $\Lambda$-mod is an additive resolving subcategory of $\Lambda$-mod which is closed under direct summands. Also by \cite[Remark 4.6]{ab}, it is a functorially finite subcategory of $\Lambda$-mod when $\Lambda$ is a virtually Gorenstein artin algebra. Moreover $\Lambda$ is a virtually Gorenstein artin algebra if and only if ${\rm GProj}(\Lambda)={\underrightarrow{\lim}}{\rm Gproj}(\Lambda)$. In this case ${\rm GProj}(\Lambda)$ is a locally finitely presented category with products (see \cite[Remark 4.6]{ab}). \\

As a consequence of Theorem \ref{3.1} we have the following result.

\begin{Cor}{\rm (See \cite[Main Theorem]{che} and \cite[Theorems 4.10 and 4.11]{ab})}\label{3.3}
Let $\Lambda$ be a virtually Gorenstein artin algebra. Then the following statements are equivalent.
\begin{itemize}
\item[$({\rm i})$] $\Lambda$ is of finite CM-type.
\item[$({\rm ii})$] Every Gorenstein projective left $\Lambda$-module is a direct sum of finitely generated module.
\item[$({\rm iii})$] Every Gorenstein projective left $\Lambda$-module is a direct sum of indecomposable module.
\item[$({\rm iv})$] There is an additive equivalence between ${\rm GProj}(\Lambda)$ and the category of projective modules over an artin algebra.
\item[$({\rm v})$] There is an additive equivalence between ${\rm Gproj}(\Lambda)$ and the category of finitely generated projective modules over an artin algebra.
\end{itemize}
\end{Cor}

\subsection{$n$-cluster tilting subcategories}
Let $\Lambda$ be an artin algebra and $n$ be a positive integer. Iyama in \cite[Definition 2.2]{I1} defined $n$-cluster tilting subcategories. A full subcategory $\mathscr{X}$ of $\Lambda$-mod is called {\it $n$-cluster tilting} if it is functorially finite and
\begin{align*}
\mathscr{X}&=\{M\in\text{$\Lambda$-mod}\,|\,\mathrm{Ext}^i_\Lambda(M,\mathscr{X})=0 \,\,\, \text{for } 0<i<n\}\\
&=\{M\in\text{$\Lambda$-mod}\,|\,\mathrm{Ext}^i_\Lambda(\mathscr{X},M)=0 \,\,\, \text{for } 0<i<n\}.
\end{align*}

Note that $\Lambda$-mod is the unique $1$-cluster tilting subcategory of $\Lambda$-mod.

An $n$-cluster tilting subcategory $\mathscr{X}$ of $\Lambda$-mod is called of {\em finite type} if the number of isomorphism classes of indecomposable objects in $\mathscr{X}$ is finite.
As a consequence of Theorem \ref{3.1} we have the following result.

\begin{Cor}{\rm (See \cite[Corollary 4.15]{EN})}\label{4.1}
Let $\Lambda$ be an artin algebra, $\mathscr{X}$ be an $n$-cluster tilting subcategory of $\Lambda$-mod and  $\mathscr{A}={\underrightarrow{\lim}}\mathscr{X}$. Then the following statements are equivalent.
\begin{itemize}
\item[$({\rm i})$] $\mathscr{X}$ is of finite type.
\item[$({\rm ii})$] Every module in $\mathscr{X}$ is a direct sum of finitely generated modules.
\item[$({\rm iii})$] $\mathscr{A}$ is pure semisimple.
\item[$({\rm iv})$] Any family of homomorphisms between indecomposable modules in $\mathscr{X}$ is noetherian.
\item[$({\rm v})$] ${\rm Mod}(\mathscr{X})$ is locally finite.
\item[$({\rm vi})$] $\mathscr{X}(X,-)$ is finite for each $X \in \mathscr{X}$.
\item[$({\rm vii})$] $\mathscr{X}(S,-)$ is finite for each simple left $\Lambda$-module $S$.
\item[$({\rm viii})$] There is an additive equivalence between $\mathscr{A}$ and the category of projective modules over an artin algebra.
\item[$({\rm ix})$] There is an additive equivalence between $\mathscr{X}$ and the category of finitely generated projective modules over an artin algebra.
\item[$({\rm x})$] Every module in $\mathscr{A}$ is a direct sum of indecomposable modules.
\end{itemize}
\end{Cor}

Ebrahimi and the second author in \cite{EN} introduced pure semisimple $n$-cluster tilting subcategories. An $n$-cluster tilting subcategory $\mathscr{X}$ of $\Lambda$-mod is called {\it pure semisimple} if ${\rm Add}(\mathscr{X})$ is an $n$-cluster tilting subcategory of $\Lambda$-Mod \cite[Definition 2.13]{EN}. Corollary \ref{4.1} and \cite[Corollary 4.15]{EN} show that an $n$-cluster tilting subcategory $\mathscr{X}$ of $\Lambda$-mod is pure semisimple if and only if ${\underrightarrow{\lim}}\mathscr{X}$ is pure semisimple as locally finitely presented category.

\section*{acknowledgements}
The research of the first author was in part supported by a grant from IPM. Also, the research of the second author was in part supported by a grant from IPM (No. 1400170417).

\end{document}